\documentclass[12pt]{article}
\usepackage{color}
\usepackage{amssymb,latexsym}
\usepackage{amsmath,amsthm} 
\usepackage{enumitem}
\newcommand{\ra}{\rightarrow}
\newcommand{\G}{\Gamma}
\newcommand{\DE}{\mathcal{D}^{\perp}}

\newcommand{\n}{\nabla}
\newcommand{\om}{\omega}
\newcommand{\w}{\wedge}

\newcommand{\te}{\theta}
\newcommand{\p}{\partial}
\newcommand{\al}{\alpha}

\newcommand{\0}{\Omega}
\newcommand{\Om}{\Omega}

\newcommand{\bJ}{\bar J}

\newcommand{\J}{\mathcal{J}}

\newcommand{\De}{\mathcal{D}}

\newtheorem{theorem}{Theorem}[section]

\theoremstyle{definition}

\newtheorem{remark}[theorem]{Remark}

\title{A DETAILED DESCRIPTION OF THE GENERALIZED CALABI TYPE K\"AHLER SURFACES}
\author{\\EWELINA MULAWA}
\date{}

\begin{document}

\maketitle

\begin{abstract}
In this paper we study QCH K\"ahler surfaces, i.e. 4-dimensional Riemannian manifolds (of
signature (++++)) admitting a K\"ahler complex structure with quasi-constant holomorphic sectional curvature. We give a detailed description of QCH K\"ahler surfaces of generalized Calabi type.
\end{abstract}



\section{Introduction}
In our paper \cite{J-M} we have given a classification of generalized Calabi type K\"ahler surfaces and in particular a classification of Calabi type K\"ahler surfaces. We classified generalized Calabi type K\"ahler QCH surfaces which admit an opposite
Hermitian structure which is not locally conformally K\"ahler. In this paper we give a detailed description of such surfaces. In every case we compute the Ricci tensor, the scalar curvature, the conformal scalar curvature, the Weyl tensor $W^-$, the sectional curvature of the distributions $\De$ and $\DE$. We give the conditions on the generalized Calabi type K\"ahler surfaces which are not of Calabi type (i.e. the opposite Hermitian structure $J$ is not conformally Kähler) to be a space forms. Those conditions hold if and only if the holomorphic sectional curvature is constant. Then the generalized Calabi type Kähler surfaces which are not of Calabi type are Einstein and are the space forms. Then they are locally isometric to $\mathbb{C}^2$ with its standard metric of constant holomorphic sectional curvature $0$ in the semi-symmetric case, $\mathbb{CP}^2$ with the Fubini-Study metric of constant holomorphic sectional curvature $4a^2$ if $\al = 2a \tan az$ and $B^2$ with its standard metric of constant holomorphic sectional curvature $-4a^2$ if $\al = -2a \coth az$ or $\al = -2a \tanh az$. We get also that the generalized Calabi type K\"ahler surfaces which are not of Calabi type are space forms if and only if $W^-=0$. If $(M,g)$ is not a space form, then $W^- \neq 0$. In his paper we get examples of local hermitian structures which are not locally conformally K\"ahler on space forms $\mathbb{CP}^2$, $\mathbb{C}^2$ and $B^2$. The K\"ahler manifolds are studied by many authors, see for example \cite{D-M-2}, \cite{D} and \cite{D-T}. The QCH surfaces are investigated in a series of works \cite{J-1}, \cite{J-2}, \cite{J-3}, \cite{J-4}, \cite{J-5}.

\section{Hermitian surfaces}

Let $(M,g,J)$ be an {\it almost Hermitian manifold}, i.e., $J$ is an  almost complex structure that is  orthogonal with respect to $g$, i.e., $g(X,Y)=g(JX,JY)$ for all $X,Y\in\frak X(M)$. We say that $(M,g,J)$ is a {\it Hermitian  manifold} if the almost Hermitian structure $J$ is integrable, which means that the Nijenhuis tensor $N^J(X,Y)$ vanishes.   This is also equivalent to integrability of the complex distribution $T^{(1,0)}M\subset TM\otimes\Bbb C$.

The {\it K\"ahler form} of $(M,g,J)$ is $\Omega (X,Y)=g(JX,Y)$. If $d\Omega=0$ the Hermitian manifold is called a {\it K\"ahler  manifold}. In the sequel we shall consider  K\"ahler manifolds of real dimension $4$ which are called {\it K\"ahler surfaces}. Such manifolds are always oriented and we choose an orientation in such a way that the K\"ahler form $\Omega(X,Y)=g(JX,Y)$ is self-dual (i.e., $\Omega \in\w^+M$). 

We investigate K\"ahler surfaces admitting an opposite Hermitian structure. A Hermitian 4-manifold $(M,g,J)$ is said to have an {\it opposite Hermitian structure} if it admits an orthogonal Hermitian structure $\bar J$ with an anti-self-dual K\"ahler form $\bar \Omega$. A Hermitian manifold $(M,g,J)$ is said to have {\it $J$-invariant Ricci tensor} $Ric$ if $Ric(X,Y)=Ric(JX,JY)$ for all $X,Y\in \frak X(M)$.

For a Hermitian surface the covariant derivative of the K\"ahler form $\Omega$ is locally expressed by
\[\nabla \Omega =a\otimes\Phi+\J a\otimes\J\Phi,\]
with $\J a(X)=-a(JX)$, $\Phi\in LM$, where $LM$ is the bundle of self-dual forms orthogonal to $\Omega$ and $\J$ is the complex structure on $LM$. The {\it Lee form} $\te$ of $(M,g,J)$ is defined by the equality $d\Omega=2\te\w \Omega$. We have $2\te=-\delta\0\circ J$. An involutive distribution is called a {\it foliation}. A foliation $\De$ is called {\it totally geodesic} if every leaf is a totally geodesic submanifold of $(M,g)$, i.e., $\n_XY\in \G(\De)$ for every $X,Y\in\G(\De)$.

We denote by $R$ the Riemann curvature tensor
\[R(X, Y)Z = ([\nabla_X , \nabla_Y ] - \nabla_{[X,Y]}) Z\]
and we write
$R(X, Y, Z, W) = g(R(X, Y)Z, W)$.

{\it Calabi type K\"ahler surfaces} are K\"ahler surfaces which admit a Hamiltonian Killing vector field $X$, such that the opposite almost Hermitian structure defined by the complex distribution $\De=span\{X,JX\}$ is Hermitian and locally conformally K\"ahler, i.e., $d \te =0$,.

{\it QCH K\"ahler surfaces} are K\"ahler surfaces $(M,g,\bJ)$ admitting a global, $2$-dimensional, $\bJ$-invariant distribution $\De$ having the following property: The holomorphic curvature $K(\pi)=R(X,\bJ X,\bJ X,X)$ of any $\bJ$-invariant $2$-plane $\pi\subset T_xM$, where $X\in \pi$ and $g(X,X)=1$, depends only on the point $x$ and the number $|X_{\De}|=\sqrt{g(X_{\De},X_{\De})}$, where $X_{\De}$ is the orthogonal projection of $X$ on $\De$ (see \cite{J-1}). In this case we have
\[R(X,\bJ X,\bJ X,X)=\phi(x,|X_{\De}|)\] 
where $\phi(x,t)=a(x)+b(x)t^2+c(x)t^4$ and
$a,b,c$ are smooth functions on $M$. Also $R=a\Pi+b\Phi+c\Psi$ for certain curvature tensors $\Pi,\Phi,\Psi\in \bigotimes^4\frak X^*(M)$ of K\"ahler type. Every QCH K\"ahler surface admits an opposite almost Hermitian structure $J$ such that the Ricci tensor of $(M,g,\bJ)$ is $J$-invariant and $(M,g,J)$ satisfies the second Gray condition:
\[R(X,Y,Z,W)-R(JX,JY,Z,W)=R(JX,Y,JZ,W)+R(JX,Y,Z,JW).\]
We denote by $\tau$ the scalar curvature of $(M,g)$, i.e., $\tau=\mathrm{tr}_g Ric$. The conformal scalar curvature $\kappa$ is defined by
\[\kappa = \tau - 6(|\te|^2+\delta \te).\] The tensor $W=W^+ + W^-$ is called the Weyl tensor and its components $W^+$, $W^-$ are called the self-dual and anti-self-dual Weyl tensors.

If the K\"ahler surface $(M, g, \bJ)$ admits a negative almost complex structure $J$ preserving the Ricci tensor $Ric$ and such that $W^-$ is degenerate with eigenvector  $\Om$ corresponding to the simple eigenvalue of $W^-$ or equivalently $J$ satisfies the second Gray condition of the curvature, then from \cite{J-1} it follows that
\begin{equation}\label{riemann}
	R= (\frac{\tau}{6} - \delta +\frac{\kappa}{12}) \Pi +(2\delta -\frac{\kappa}{2}) \Phi + \frac{\kappa}{2} \Psi,
\end{equation}
where $\delta$ is half of the difference between two eigenvalues of the Ricci tensor,
$\Pi$ is the standard K\"ahler tensor of constant holomorphic curvature, i.e.
\begin{align*}
	\Pi(X, Y, Z, U) &= \frac{1}{4} (g(Y, Z)g(X, U) - g(X, Z)g(Y, U)+g(\bJ Y, Z)g(\bJ X, U)\\
	&- g(\bJ X, Z)g(\bJ Y, U) - 2g(\bJ X, Y )g(\bJ Z, U)),
\end{align*}
the tensor $\Phi$ is defined by the following relation
\begin{align*}
	\Phi(X, Y, Z, U) &= \frac{1}{8} (g(Y, Z)h(X, U) - g(X, Z)h(Y, U) +g(X, U)h(Y, Z)- g(Y, U)h(X, Z)\\
	&+ g(\bJ Y, Z)h(\bJ X, U) -g(\bJ X, Z)h(\bJ Y, U) + g(\bJ X, U)h(\bJ Y, Z)\\
	&-g(\bJ Y, U)h(\bJ X, Z)-2g(\bJ X, Y )h(\bJ Z, U) - 2g(\bJ Z, U)h(\bJ X, Y ))
\end{align*}
and finally
\[\Psi(X, Y, Z, U) = -h(\bJ X, Y )h(\bJ Z, U) = -(h_{\bJ} \otimes h_{\bJ} )(X, Y, Z, U),\]
where $h_{\bJ} (X,Y)=h(\bJ X,Y)$, $h=g \circ (p_{\DE} \times p_{\DE})$, $p_{\DE}$ is the orthogonal projection on $\DE$.

In \cite{J-1} it is proved that every QCH K\"ahler surface $(M,g,\bJ)$ with opposite almost Hermitian structure $J$, which is Hermitian and locally conformally K\"ahler is of Calabi type or is orthotoric or is hyperk\"ahler.    If $J$ is determined by a foliation into complex curves then $(M,g,\bJ)$ is of Calabi type.

By a {\it generalized Calabi type K\"ahler surface} we maean a QCH K\"ahler surface such that the opposite almost Hermitian structure $J$ is determined by a complex foliation into complex curves and is Hermitian. It just means that $\De$ is integrable and $J$ is also integrable.

In the sequel we shall only consider generalized Calabi type K\"ahler surfaces $(M,g,\bJ)$ whose opposite Hermitian structure $J$ is not K\"ahler.

\section{Properties of Hermitian surfaces with Hermitian Ricci tensor and K\"ahler natural opposite structure}

Let $(M,g,J)$ be a Hermitian 4-manifold. Assume that  $|\n J| > 0$ on $M$. Then from \cite{J-5} there exists a global oriented orthonormal basis $\{E_3,E_4\}$ of the nullity foliation $\De$, such that for any local orthonormal oriented basis $\{E_1,E_2\}$ of $\DE$ 
\[\n \0=\al(\te_1\otimes\Phi+\te_2\otimes\Psi),\]
where $\Phi=\te_1 \w \te_3 -\te_2 \w \te_4, \Psi=\te_1 \w \te_4+\te_2 \w \te_3$, $\al=\pm \frac1{2\sqrt{2}}|\n J|$ and $\{\te_1,\te_2,\te_3,\te_4\}$ is a cobasis dual to $\{E_1,E_2,E_3,E_4\}$. Moreover $\delta\0=-2\al \te_3$ and $\te=-\al \te_4$.
A frame $\{E_1, E_2, E_3, E_4\}$ we call a special frame for $(M, g, J)$.
If the natural opposite almost Hermitian structure is Hermitian then $\De$ is totally geodesic (see \cite{J-4}).
If $(M,g,J)$ is a Hermitian surface with $|\n J|> 0$ on $M$, then the distributions $\De,\DE$ define a natural opposite almost Hermitian structure $\bar{J}$ on $M$. This structure is defined as follows $\bar J|_{\De}=-J|_{\De},\bar J|_{\DE}=J|_{\DE}$. In the special basis we just have: $\bar J E_1=E_2, \bar J  E_3=-E_4$.

Assume that $(M,g,J)$ is a Hermitian surface with Hermitian Ricci tensor and K\"ahler natural opposite structure $\bJ$ and assume that $\te_1=fdx,\te_2=fdy$. We can always assume that $\te_1=fdx,\te_2=fdy$, where  $(x,y,z,t)$  is a local foliated coordinate system i.e. $\De=\mathrm{ker} dx\cap \mathrm{ker} dy$ (see \cite{J-M}). Then from \cite{J-5} we have 
\begin{equation}\label{nawiasy}
	\begin{aligned}
	[E_1,E_4]&=-\frac\al 2 E_1+E_2\ln\al E_3,\\
	[E_2,E_4]&=-\frac\al 2E_2-E_1\ln\al E_3,\\
	[E_1,E_3]&=-E_2\ln\al E_4,\\
	[E_2,E_3]&=E_1\ln\al E_4,\\
	[E_3,E_4]&=-(-E_4\ln\al+\al)E_3,\\
	[E_1,E_2]&=\G^1_{12}E_1-\G^2_{21}E_2+\al E_3.
	\end{aligned}
\end{equation}
We also have
\begin{equation}\label{dteta}
	\begin{aligned}
	d\te_1&=\G^2_{11}\te_1\w\te_2+\frac12\al\te_1\w\te_4,\\
	d\te_2&=-\G^1_{22}\te_1\w\te_2+\frac12\al\te_2\w\te_4,\\
	d\te_3&=-\al\te_1\w\te_2-E_2\ln\al\te_1\w\te_4+E_1\ln\al\te_2\w\te_4+(-E_4\ln\al+\al)\te_3\w\te_4,\\
	d\te_4&=E_2\ln\al\te_1\w\te_3-E_1\ln\al\te_2\w\te_3.
	\end{aligned}
\end{equation}
From \cite{J-M} we conclude that
\begin{equation}\label{omegi}
	\begin{aligned}
	\om^3_1 &= \om^2_4 = -\frac{1}{2} \al \te_2,\\
	\om^1_4 &= \om^2_3 = -\frac{1}{2} \al \te_1,\\
	\om^2_1 &= \G^2_{11} \te_1 - \G^1_{22} \te_2 - \frac{\al}{2} \te_3,\\
	\om^3_4 &= E_2 \ln \al \te_1 - E_1 \ln \al \te_2 + (E_4 \ln \al - \al) \te_3.
	\end{aligned}
\end{equation}
Now using (\ref{nawiasy}) and (\ref{omegi}), we obtain
\begin{equation}\label{R12}
	\begin{aligned}
		(R(E_1, E_2) E_1, E_2) &= d \om^2_1 (E_1, E_2) + (\om^2_3 \w \om^3_1 + \om^2_4 \w \om^4_1)(E_1, E_2)\\
		&= E_1 \om^2_1 (E_2) - E_2 \om^2_1 (E_1) - \om^2_1 ([E_1, E_2]) + \frac{1}{2} \al^2\\
		&= E_1 \om^2_1 (E_2) - E_2 \om^2_1 (E_1) - \om^2_1 (\G^1_{22} E_1 - \G^2_{21} E_2 + \al E_3) + \frac{1}{2} \al^2\\
		&= -E_1 \G^1_{22} - E_2 \G^2_{11} + (\G^2_{11})^2 + (\G^1_{22})^2 + \al^2
	\end{aligned}
\end{equation}
	and
\begin{equation}\label{R34}
	\begin{aligned}
		(R(E_3, E_4)E_3, E_4) &= d \om^4_3 (E_3, E_4) + (\om^4_1 \w \om^1_3 + \om^4_2 \w \om^2_3)(E_3, E_4)\\
		&= E_3 \om^4_3 (E_4) - E_4 \om^4_3 (E_3) - \om^4_3 ([E_3, E_4])\\
		&= E_3 \om^4_3 (E_4) - E_4 \om^4_3 (E_3) - \om^4_3 ((E_4 \ln \al - \al) E_3).
	\end{aligned}
\end{equation}
This implies that the sectional curvature
\begin{equation}\label{k12}
	\begin{aligned}
K(E_1,E_2)&=R(E_1, E_2, E_2, E_1) = -(R(E_1, E_2) E_1, E_2)\\
&= E_1 \G^1_{22} + E_2 \G^2_{11} - (\G^2_{11})^2 - (\G^1_{22})^2 - \al^2
	\end{aligned}
\end{equation}
and
\begin{equation}\label{k34}
	\begin{aligned}
K(E_3,E_4)&=R(E_3, E_4, E_4, E_3) = -(R(E_3, E_4) E_3, E_4)\\
&= -E_3 \om^4_3 (E_4) + E_4 \om^4_3 (E_3) + \om^4_3 ((E_4 \ln \al - \al) E_3).
	\end{aligned}
\end{equation}
For
	\[\te = -\al \te_4\]
we have
	\[\vert \te \vert^2 = \al^2 \]
and using (\ref{omegi}), we obtain
	\begin{align*}
		\delta \te &= -E_i \lrcorner \n _{E_i} \te = -E_i \lrcorner (E_i(-\al) \te_4 - \al \n _{E_i} \te_4 = -E_i \lrcorner (E_i(-\al) \te_4 - \al \sum_j \om^j_4 (E_i) \te_j)\\
		&= E_4 \al + \al (\om^1_4(E_1) + \om^2_4(E_2) +\om^3_4(E_3)) = 2 (E_4 \al -\al^2).  
	\end{align*}
Hence
\begin{equation}\label{teta2}
	\vert \te \vert^2 +\delta \te = 2 E_4 \al -\al^2.
\end{equation}
Now using (\ref{nawiasy}) and (\ref{omegi}), we get the Ricci tensor
	\begin{align*}
		Ric(E_1, E_1) &= \sum_{i=1}^{4} (R(E_i,E_1)E_1,E_i) = \sum_{i=1}^{4} (d \om^i_1 (E_i,E_1)+\om^i_k \w \om^k_1 (E_i,E_1))\\
		&= E_2 \G^2_{11} + E_1 \G^1_{22} - (\G^2_{11})^2 - (\G^1_{22})^2 +E_4 \al - \frac{3}{2} \al^2,\\
		Ric(E_1, E_2) &= \sum_{i=1}^{4} (R(E_i,E_1)E_2,E_i) = \sum_{i=1}^{4} (d \om^i_2 (E_i,E_1)+\om^i_k \w \om^k_2 (E_i,E_1))\\
		&= E_3 \frac{\al}{2},\\
		Ric(E_1, E_3) &= \sum_{i=1}^{4} (R(E_i,E_1)E_3,E_i) = \sum_{i=1}^{4} (d \om^i_3 (E_i,E_1)+\om^i_k \w \om^k_3 (E_i,E_1))\\
		&= \frac{3}{2} E_2 \al - E_4(E_2 \ln \al) - E_2 \ln \al (E_4 \ln \al),\\
		Ric(E_1, E_4) &= \sum_{i=1}^{4} (R(E_i,E_1)E_4,E_i) = \sum_{i=1}^{4} (d \om^i_4 (E_i,E_1)+\om^i_k \w \om^k_4 (E_i,E_1))\\
		&= E_3(E_2 \ln \al) - E_1 (E_4 \ln \al - \al),\\
\end{align*}

\begin{equation}\label{Ricci}
	\begin{aligned}
Ric(E_2, E_2) &= \sum_{i=1}^{4} (R(E_i,E_2)E_2,E_i) = \sum_{i=1}^{4} (d \om^i_2 (E_i,E_2)+\om^i_k \w \om^k_2 (E_i,E_2))\\
&= E_2 \G^2_{11} + E_1 \G^1_{22} - (\G^2_{11})^2 - (\G^1_{22})^2 +E_4 \al - \frac{3}{2} \al^2 = Ric(E_1,E_1),\\
Ric(E_2, E_3) &= \sum_{i=1}^{4} (R(E_i,E_2)E_3,E_i) = \sum_{i=1}^{4} (d \om^i_3 (E_i,E_2)+\om^i_k \w \om^k_3 (E_i,E_2))\\
&= -\frac{3}{2} E_1 \al + E_4(E_1 \ln \al) + E_1 \ln \al (E_4 \ln \al),\\
Ric(E_2, E_4) &= \sum_{i=1}^{4} (R(E_i,E_2)E_4,E_i) = \sum_{i=1}^{4} (d \om^i_4 (E_i,E_2)+\om^i_k \w \om^k_4 (E_i,E_2)) \\
&= - E_3(E_1 \ln \al) - E_2 (E_4 \ln \al - \al),\\
Ric(E_3, E_3) &= \sum_{i=1}^{4} (R(E_i,E_3)E_3,E_i) = \sum_{i=1}^{4} (d \om^i_3 (E_i,E_3)+\om^i_k \w \om^k_3 (E_i,E_3))\\
&= - E_4(E_4 \ln \al - \al) - (E_4 \ln \al - \al)^2 +\al (E_4 \ln \al - \al) + \frac{\al^2}{2},\\
Ric(E_3, E_4) &= \sum_{i=1}^{4} (R(E_i,E_3)E_4,E_i) = \sum_{i=1}^{4} (d \om^i_4 (E_i,E_3)+\om^i_k \w \om^k_4 (E_i,E_3)) \\
&= E_3 \al,\\
Ric(E_4, E_4) &= \sum_{i=1}^{4} (R(E_i,E_4)E_4,E_i) = \sum_{i=1}^{4} (d \om^i_4 (E_i,E_4)+\om^i_k \w \om^k_4 (E_i,E_4))\\
&= - E_4(E_4 \ln \al - \al) - (E_4 \ln \al - \al)^2 +\al (E_4 \ln \al - \al) + \frac{\al^2}{2} \\
&= Ric (E_3,E_3).\\
\end{aligned}
\end{equation}

\section{A description of the generalized Calabi type K\"ahler surfaces which are not of Calabi type}

The semi-symmetric generalized Calabi type K\"ahler surfaces are classified in \cite{J-5} (see also \cite{J-2}) and the remaining cases are classified in \cite{J-M}.
For the convenience of the reader, we recall the main results from those papers. First we have the semi-symmetric case.

\begin{theorem}\label{tw:semi-sym}
	Let  $U\subset \mathbb R^2$ and let $g_{\Sigma}=h^2(dx^2+dy^2)$ be a Riemannian metric on $U$, where $h:U \ra \mathbb R$ is a positive function $h=h(x,y)$.  Let $\om_{\Sigma}=h^2dx \w dy$ be the volume form of $\Sigma=(U,g)$. Let $M=U \times N$, where $N=\{(z,t) \in \mathbb R^2: z < 0\}$. Define a metric $g$ on $M$ by $g(X,Y)=z^2g_{\Sigma}(X,Y)+ \te_3(X)\te_3(Y)+ \te_4(X)\te_4(Y)$, where
	\begin{align*}
		\te_3&=-\frac{z}{2}dt+(\cos \frac{1}{2}tH(x,y)+zl_2(x,y))dx\\
		&+(-\sin \frac{1}{2}tH(x,y)+zn_2(x,y))dy,\\
		\te_4&=dz-\sin \frac{1}{2}tH(x,y)dx-\cos \frac{1}{2}tH(x,y)dy
	\end{align*}
	and the function $H$ satisfies the equation $\Delta \ln H=(\ln H)_{xx}+(\ln H)_{yy}=2h^2$ on $U$, $l_2=-(\ln H)_y, n_2=(\ln H)_x$. Then $(M,g)$ admits a K\"ahler structure $\bar J$ with the K\"ahler form $\bar \Omega =z^2\om_{\Sigma}+\te_4\w \te_3$  and a Hermitian structure $J$ with the K\"ahler form $\Omega =z^2\om_{\Sigma}+\te_3\w \te_4$. The Ricci tensor of $(M,g)$ is $J$-invariant and $J$ is not locally conformally K\"ahler. The Lee form of $(M,g,J)$ is $\te=-\al\te_4$, where $\al=-\frac{2}{z}$. The scalar curvature of $(M,g)$ is $\tau=-\frac{2 \Delta \ln h + 8h^2}{z^2 h^2}$ and is equal to conformal curvature $\kappa$ of $(M,g,J)$. $(M,g,\bar \Omega)$ is a semi-symmetric QCH K\"ahler surface. 
\end{theorem}
Next we recall the remaining cases.
\begin{theorem}\label{tw:tan}
	Let  $U\subset \mathbb R^2$ and let $g_{\Sigma}=h^2(dx^2+dy^2)$ be a Riemannian metric on $U$, where $h:U \ra \mathbb R$ is a positive function $h=h(x,y)$.  Let $\om_{\Sigma}=h^2dx \w dy$ be the volume form of $\Sigma=(U,g)$. Let $M=U \times N$, where $N=\{(z,t) \in \mathbb R^2: \vert z \vert <\frac{\pi}{2 \vert a \vert}\}$. Define a metric $g$ on $M$ by $g(X,Y)=(\cos az)^2g_{\Sigma}(X,Y)+ \te_3(X)\te_3(Y)+ \te_4(X)\te_4(Y)$, where
	\begin{align*}
		\te_3&=\sin 2azdt-(\cos 2at\cos 2azH(x,y)+\sin 2azl_2(x,y))dx\\
		&-(-\sin 2at\cos 2azH(x,y)+\sin 2azn_2(x,y))dy,\\
		\te_4&=dz-\sin 2atH(x,y)dx-\cos 2atH(x,y)dy
	\end{align*}
	and the function $H$ satisfies the equation  $\Delta \ln H=(\ln H)_{xx}+(\ln H)_{yy}=2a^2 h^2-4a^2 H^2$ on $U$, $l_2=-\frac{1}{2a}(\ln H)_y, n_2=\frac{1}{2a}(\ln H)_x$. Then $(M,g)$ admits a K\"ahler structure $\bar J$ with the K\"ahler form $\bar \Omega =(\cos az)^2\om_{\Sigma}+\te_4\w \te_3$  and a Hermitian structure $J$ with the K\"ahler form $\Omega =(\cos az)^2\om_{\Sigma}+\te_3\w \te_4$. The Ricci tensor of $(M,g)$ is $J$-invariant and $J$ is not locally conformally K\"ahler. The Lee form of $(M,g,J)$ is $\te=-\al\te_4$, where $\al=2a\tan az$.
\end{theorem}

\begin{theorem}\label{tw:coth}
	Let  $U\subset \mathbb R^2$ and let $g_{\Sigma}=h^2(dx^2+dy^2)$ be a Riemannian metric on $U$, where $h:U \ra \mathbb R$ is a positive function $h=h(x,y)$.  Let $\om_{\Sigma}=h^2dx \w dy$ be the volume form of $\Sigma=(U,g)$. Let $M=U \times N$, where $N=\{(z,t) \in \mathbb R^2: z < 0\}$. Define a metric $g$ on $M$ by $g(X,Y)=(\sinh az)^2g_{\Sigma}(X,Y)+ \te_3(X)\te_3(Y)+ \te_4(X)\te_4(Y)$, where
	\begin{align*}
		\te_3&=\sinh 2azdt-(-\sin 2at\cosh 2azH(x,y)+\sinh 2azl_2(x,y))dx\\
		&-(\cos 2at\cosh 2azH(x,y)+\sinh 2azn_2(x,y))dy,\\
		\te_4&=dz-\cos 2atH(x,y)dx-\sin 2atH(x,y)dy
	\end{align*}
	and the function $H$ satisfies the equation  $\Delta \ln H=(\ln H)_{xx}+(\ln H)_{yy}=2a^2 h^2+4a^2 H^2$ on $U$, $l_2=\frac{1}{2a}(\ln H)_y, n_2=-\frac{1}{2a}(\ln H)_x$. Then $(M,g)$ admits a K\"ahler structure $\bar J$ with the K\"ahler form $\bar \Omega =(\sinh az)^2\om_{\Sigma}+\te_4\w \te_3$  and a Hermitian structure $J$ with the K\"ahler form $\Omega =(\sinh az)^2\om_{\Sigma}+\te_3\w \te_4$. The Ricci tensor of $(M,g)$ is $J$-invariant and $J$ is not locally conformally K\"ahler. The Lee form of $(M,g,J)$ is $\te=-\al\te_4$, where $\al=-2a\coth az$.
\end{theorem}

\begin{theorem}\label{tw:tanh}
	Let  $U\subset \mathbb R^2$ and let $g_{\Sigma}=h^2(dx^2+dy^2)$ be a Riemannian metric on $U$, where $h:U \ra \mathbb R$ is a positive function $h=h(x,y)$.  Let $\om_{\Sigma}=h^2dx \w dy$ be the volume form of $\Sigma=(U,g)$. Let $M=U \times N$, where $N=\{(z,t) \in \mathbb R^2: z < 0\}$. Define a metric $g$ on $M$ by $g(X,Y)=(\cosh az)^2g_{\Sigma}(X,Y)+ \te_3(X)\te_3(Y)+ \te_4(X)\te_4(Y)$, where
	\begin{align*}
		\te_3&=\sinh 2azdt-(\cos 2at\cosh 2azH(x,y)+\sinh 2azl_2(x,y))dx\\
		&-(-\sin 2at\cosh 2azH(x,y)+\sinh 2azn_2(x,y))dy,\\
		\te_4&=dz-\sin 2atH(x,y)dx-\cos 2atH(x,y)dy
	\end{align*}
	and the function $H$ satisfies the equation  $\Delta \ln H=(\ln H)_{xx}+(\ln H)_{yy}=-2a^2 h^2+4a^2 H^2$ on $U$, $l_2=-\frac{1}{2a}(\ln H)_y, n_2=\frac{1}{2a}(\ln H)_x$. Then $(M,g)$ admits a K\"ahler structure $\bar J$ with the K\"ahler form $\bar \Omega =(\cosh az)^2\om_{\Sigma}+\te_4\w \te_3$  and a Hermitian structure $J$ with the K\"ahler form $\Omega =(\cosh az)^2\om_{\Sigma}+\te_3\w \te_4$. The Ricci tensor of $(M,g)$ is $J$-invariant and $J$ is not locally conformally K\"ahler. The Lee form of $(M,g,J)$ is $\te=-\al\te_4$, where $\al=-2a\tanh az$.
\end{theorem}

Now we give a detailed description of such surfaces which are the generalized Calabi type K\"ahler surfaces and not of Calabi type.

First we have the semi-symmetric case. 

\begin{theorem}
	If metric $g$ is as in Theorem \ref{tw:semi-sym} ($\al = -\frac{2}{z}$), then
	\begin{enumerate}[label=\upshape(\roman*), leftmargin=*, widest=iiii]
		\item $\Delta \ln H=2h^2$ on $U$, \label{it:1}
		\item $\G^2_{11} = -\frac{\cos \frac{1}{2} t H}{z^2 h}-\frac{h_y}{z h^2}$ and $\G^1_{22} = -\frac{\sin \frac{1}{2} t H}{z^2 h}-\frac{h_x}{z h^2}$, \label{it:2}
		\item the Ricci tensor of $(M,g)$ represented in the frame $\{E_1,E_2,E_3,E_4\}$ is 
		\[
		Ric=\begin{bmatrix}
			-\frac{\Delta \ln h}{z^2 h^2} -\frac{4}{z^2} & 0 & 0 & 0 \\
			0 & -\frac{\Delta \ln h}{z^2 h^2} -\frac{4}{z^2} & 0 & 0\\
			0 & 0 & 0 & 0\\
			0 & 0 & 0 & 0
		\end{bmatrix},\]\label{it:3}
		\item the scalar curvature of $(M,g)$ is $\tau = -\frac{2 \Delta \ln h}{z^2 h^2} -\frac{8}{z^2}$,\label{it:4}
		\item the conformal scalar curvature of $(M,g)$ is $\kappa = -\frac{2 \Delta \ln h}{z^2 h^2} -\frac{8}{z^2}$,\label{it:5}
		\item the Weyl tensor \[W^- =
		\begin{bmatrix}
			-\frac{\Delta \ln h}{3z^2 h^2} -\frac{4}{3z^2} & 0 & 0 \\
			0 & \frac{\Delta \ln h}{6z^2 h^2} +\frac{2}{3z^2} & 0\\
			0 & 0 & \frac{\Delta \ln h}{6z^2 h^2} +\frac{2}{3z^2}  
		\end{bmatrix}\] in the frame $\{\Om, \phi, \psi \}$, where $\{\phi, \psi \}$ is any orthonormal basis of $LM$,\label{it:6}
		\item the sectional curvature of $(M,g)$ of $\DE$ is
		$K(E_1,E_2)= -\frac{\Delta \ln h}{z^2 h^2} - \frac{4}{z^2}$
		and of $\De$ is
		$K(E_3,E_4)= 0$,\label{it:7}
		\item $(M,g)$ is a space form (it is locally isometric to $\mathbb{C}^2$ with its standard metric of constant holomorphic sectional curvature $0$) if and only if $\Delta \ln h = -4 h^2$.\label{it:8}
	\end{enumerate}
\end{theorem}
\begin{proof}
	\ref{it:1} It follows from Theorem \ref{tw:semi-sym} that $\Delta \ln H=2h^2$ on $U$.\\
	\ref{it:2} Take a coordinate system such that $E_1=\frac1f\p_x+k\p_z+l\p_t,E_2=\frac1f\p_y+m\p_z+n\p_t,E_3=\al \p_t,E_4=\p_z$. Then  $\te_1=fdx, \te_2=fdy, \te_4=dz-(fk)dx-(fm)dy,\te_3=\frac{1}{\al} dt-(\frac{1}{\al} lf)dx-(\frac{1}{\al} nf)dy$.  Let  $\al=-\frac{2}{z}$. Then \[f=zh, m=\frac{\cos \frac{1}{2} t H}{zh}, k=\frac{\sin \frac{1}{2} t H}{zh},\]
	\[n=\frac{-\frac{2}{z} \sin \frac{1}{2} t H +2 (\ln H)_x}{zh},
	l=\frac{\frac{2}{z} \cos \frac{1}{2} t H -2 (\ln H)_y}{zh}.\]
	Since
	\[\te_1 = f dx = zh dx,\]
	we have
	\[d \te_1 = zh_y dy \w dx + h dz \w dx.\]
	On the other hand, from (\ref{dteta}) we have
	\begin{align*}
		d \te_1 &= \G^2_{11} \te_1 \w \te_2 + \frac{1}{2} \al \te_1 \w \te_4\\
		&= \G^2_{11} z^2 h^2 dx \w dy  + \frac{1}{2} \al (f dx) \w (dz - \sin \frac{1}{2} t H dx - \cos \frac{1}{2} t H dy)\\
		&= (\G^2_{11} z^2 h^2 - \frac{1}{2} \al z h \cos \frac{1}{2} t H ) dx \w dy  + \frac{1}{2} \al zh dx \w dz.
	\end{align*}
	This yields
	\begin{align*}
		\G^2_{11} &= \frac{\frac{1}{2} \al zh \cos \frac{1}{2} t H -zh_y}{z^2 h^2} = -\frac{\cos \frac{1}{2} t H}{z^2 h}-\frac{h_y}{z h^2}.
	\end{align*}
	Similarly
	\[\te_2 = f dy = zh dy\]
	and
	\[d \te_2 = zh_x dx \w dy + h dz \w dy.\]
	On the other hand, from (\ref{dteta}) we have 
	\begin{align*}
		d \te_2 &= -\G^1_{22} \te_1 \w \te_2 + \frac{1}{2} \al \te_2 \w \te_4\\
		&= -\G^1_{22} z^2 h^2 dx \w dy  - \frac{1}{z} (f dy) \w (dz - \sin \frac{1}{2} t H dx - \cos \frac{1}{2} t H dy)\\
		&= (-\G^1_{22} z^2 h^2 - h \sin \frac{1}{2} t H) dx \w dy  - h dy \w dz
	\end{align*}
	and we get
	\begin{align*}
		\G^1_{22} &= \frac{-h \sin \frac{1}{2} t H -zh_x}{z^2 h^2} = -\frac{\sin \frac{1}{2} t H}{z^2 h}-\frac{h_x}{z h^2}.
	\end{align*}
	\ref{it:3} From (\ref{Ricci}) we obtain
	\begin{align*}
		Ric(E_1, E_1) &= Ric(E_2, E_2) = -\frac{\Delta \ln h}{z^2 h^2} -\frac{4}{z^2}
	\end{align*}
	and $Ric(E_i,E_j)=0$ for the other cases. Hence in the frame $\{E_1,E_2,E_3,E_4\}$ we have
	\[
	Ric=\begin{bmatrix}
		-\frac{\Delta \ln h}{z^2 h^2} -\frac{4}{z^2} & 0 & 0 & 0 \\
		0 & -\frac{\Delta \ln h}{z^2 h^2} -\frac{4}{z^2} & 0 & 0\\
		0 & 0 & 0 & 0\\
		0 & 0 & 0 & 0
	\end{bmatrix}.\]\\
	\ref{it:4} The scalar curvature of $(M,g)$ is
	\[\tau = \mathrm{tr}_g Ric = -\frac{2 \Delta \ln h}{z^2 h^2} -\frac{8}{z^2}.\]
	\ref{it:5} Since by (\ref{teta2})
	\[\vert \te \vert^2 +\delta \te = 2 E_4 \al - \al^2 = 0,\]
	it follows that the conformal scalar curvature of $(M,g)$ is
	\[\kappa = \tau -6(\vert \te \vert^2 +\delta \te) = -\frac{2 \Delta \ln h}{z^2 h^2} -\frac{8}{z^2}.\]
	\ref{it:6} Since the Weyl tensor 
	\[W^- =
	\begin{bmatrix}
		\frac{\kappa}{6} & 0 & 0\\
		0 & -\frac{\kappa}{12} & 0\\ 
		0 & 0 & -\frac{\kappa}{12}\\ 
	\end{bmatrix},\] we have
	\[W^- =
	\begin{bmatrix}
		-\frac{\Delta \ln h}{3z^2 h^2} -\frac{4}{3z^2} & 0 & 0 \\
		0 & \frac{\Delta \ln h}{6z^2 h^2} +\frac{2}{3z^2} & 0\\
		0 & 0 & \frac{\Delta \ln h}{6z^2 h^2} +\frac{2}{3z^2} 
	\end{bmatrix}\] in the frame $\{\Om, \phi, \psi \}$, where $\{\phi, \psi \}$ is any orthonormal basis of $LM$.\\
	\ref{it:7} From (\ref{R12}) we have
	\begin{align*}
		(R(E_1, E_2) E_1, E_2) &= -E_1 \G^1_{22} - E_2 \G^2_{11} + (\G^2_{11})^2 + (\G^1_{22})^2 + \al^2\\
		&= \frac{\Delta \ln h}{z^2 h^2} + \al^2
	\end{align*}
	and from (\ref{R34}) we obtain
	\[(R(E_3, E_4) E_3, E_4) = E_3 \om^4_3 (E_4) - E_4 \om^4_3 (E_3) - \om^4_3 ((E_4 \ln \al - \al) E_3) = 0.\]
	This implies that the sectional curvature of $\DE$ is
	\begin{align*}
		K(E_1,E_2)&=R(E_1, E_2, E_2, E_1) = -(R(E_1, E_2) E_1, E_2) \\
		&= -\frac{\Delta \ln h}{z^2 h^2} - \al^2 = -\frac{\Delta \ln h}{z^2 h^2} - \frac{4}{z^2}
	\end{align*} 
	and of $\De$ is
	\[K(E_3,E_4)=R(E_3, E_4, E_4, E_3) = -(R(E_3, E_4) E_3, E_4) = 0.\]
	\ref{it:8} From \ref{it:7} it follows that
	\[K(E_1, E_2) = K(E_3, E_4) = 0\]
	if and only if
	\begin{equation}\label{deltasemi}
		\Delta \ln h = -4 h^2.
	\end{equation}
	It follows from \ref{it:1} that in this case we have
	\[\begin{cases}
		\Delta \ln h &= -4 h^2,\\
		\Delta \ln H &= 2 h^2.
	\end{cases}\]
	If (\ref{deltasemi}) holds, then $\delta=0$, $\kappa=0$ and $\tau = 0$. Then from (\ref{riemann}) we obtain that the Riemann curvature tensor $R=0$ and hence the holomorphic sectional curvature is $H=0$, since for $X$ such that $\vert \vert X \vert \vert = 1$ we have
	\[H(X) = K(X,\bJ X) = R(X,\bJ X,\bJ X,X) = 0.\]
	It follows that (\ref{deltasemi}) holds if and only if the holomorphic sectional curvature is constant and is equal to $0$.
	Then we conclude that $(M,g)$ is Einstein and is a space form if and only if $\Delta \ln h = -4 h^2$, because it is generalized Calabi type K\"ahler surface which is not of Calabi type and it has constant holomorphic sectional curvature. It is locally isometric to $\mathbb{C}^2$ with its standard metric of constant holomorphic sectional curvature $0$ (see \cite{J-M}).
\end{proof}

Now we consider the remaining cases.
\begin{theorem}
If metric $g$ is as in Theorem \ref{tw:tan} ($\al=2a\tan az$, $a \in \mathbb{R}$, $a \neq 0$), then
	\begin{enumerate}[label=\upshape(\roman*), leftmargin=*, widest=iiii]
	\item $\Delta \ln H=2a^2 h^2-4a^2 H^2$ on $U$, \label{it:1}
	\item $\G^2_{11} = \frac{a \sin az \cos 2at H}{\cos^2 az h}-\frac{h_y}{\cos az h^2}$ and $\G^1_{22} = \frac{a \sin az \sin 2at H}{\cos^2 az h} - \frac{h_x}{\cos az h^2}$, \label{it:2}
	\item the Ricci tensor of $(M,g)$ represented in the frame $\{E_1,E_2,E_3,E_4\}$ is 
	\[
	Ric=\begin{bmatrix}
		\frac{-4a^2 h^2 + 2a^2 H^2 -\Delta \ln h}{h^2 \cos^2 az} +6a^2 & 0 & 0 & 0 \\
		0 & \frac{-4a^2 h^2 + 2a^2 H^2 -\Delta \ln h}{h^2 \cos^2 az} +6a^2 & 0 & 0\\
		0 & 0 & 6a^2 & 0\\
		0 & 0 & 0 & 6a^2 
	\end{bmatrix},\]\label{it:3}
	\item the scalar curvature of $(M,g)$ is $\tau = \frac{-8a^2 h^2 + 4a^2 H^2 -2\Delta \ln h}{h^2 \cos^2 az} +24a^2$,\label{it:4}
	\item the conformal scalar curvature of $(M,g)$ is $\kappa = \frac{-8a^2 h^2 + 4a^2 H^2 -2\Delta \ln h}{h^2 \cos^2 az}$,\label{it:5}
	\item the Weyl tensor
	\[W^- =
	\begin{bmatrix}
		\frac{-4a^2 h^2 + 2a^2 H^2 -\Delta \ln h}{3 h^2 \cos^2 az} & 0 & 0 \\
		0 & \frac{4a^2 h^2 -2a^2 H^2 +\Delta \ln h}{6 h^2 \cos^2 az} & 0\\
		0 & 0 & \frac{4a^2 h^2 -2a^2 H^2 +\Delta \ln h}{6 h^2 \cos^2 az}
	\end{bmatrix}\] in the frame $\{\Om, \phi, \psi \}$, where $\{\phi, \psi \}$ is any orthonormal basis of $LM$,\label{it:6}
	\item the sectional curvature of $(M,g)$ of $\DE$ is
	$K(E_1,E_2)= \frac{2a^2 H^2 - \Delta \ln h}{\cos^2 az h^2} - 4a^2 \tan^2 az$
	and of $\De$ is
	$K(E_3,E_4)= 4a^2$,\label{it:7}
	\item $(M,g)$ is a space form (it is locally isometric to $\mathbb{CP}^2$ with the Fubini-Study metric of constant holomorphic sectional curvature $4a^2$) if and only if $\Delta \ln h = -4a^2 h^2 + 2 a^2 H^2$.\label{it:8}
\end{enumerate}
\end{theorem}
\begin{proof}
	\ref{it:1} It follows from Theorem \ref{tw:tan} that $\Delta \ln H=2a^2 h^2-4a^2 H^2$ on $U$.\\
	\ref{it:2} Take a coordinate system such that $E_1=\frac1f\p_x+k\p_z+l\p_t,E_2=\frac1f\p_y+m\p_z+n\p_t,E_3=\frac1\beta\p_t,E_4=\p_z$. Then  $\te_1=fdx, \te_2=fdy, \te_4=dz-(fk)dx-(fm)dy,\te_3=\beta dt-(\beta lf)dx-(\beta nf)dy$.  Let  $\al=2a\tan az,\beta=\sin 2az$. Let $g=\cos az$. Then \[f=gh, m=\frac{\cos 2at H}{gh}, k=\frac{\sin 2at H}{gh},\]
	\[n=\frac{-\cot 2az \sin 2at H +\frac{1}{2a} (\ln H)_x}{gh},
	l=\frac{\cot 2az \cos 2at H -\frac{1}{2a} (\ln H)_y}{gh}.\]
	Since
	\[\te_1 = f dx = gh dx,\]
	we have
	\[d \te_1 = gh_y dy \w dx + g_z h dz \w dx.\]
	On the other hand, from (\ref{dteta}) we have
	\begin{align*}
		d \te_1 &= \G^2_{11} \te_1 \w \te_2 + \frac{1}{2} \al \te_1 \w \te_4\\
		&= \G^2_{11} g^2 h^2 dx \w dy  + \frac{1}{2} \al (f dx) \w (dz - fk dx - fm dy)\\
		&= (\G^2_{11} g^2 h^2 - \frac{1}{2} \al g h \cos 2at H) dx \w dy  + \frac{1}{2} \al gh dx \w dz.
	\end{align*}
	This yields
	\begin{align*}
		\G^2_{11} &= \frac{1}{2} \al \frac{\cos 2at H}{gh} - \frac{h_y}{gh^2} = \frac{1}{2} \al \frac{\cos 2at H}{\cos az h}-\frac{h_y}{\cos az h^2}\\
		&= \frac{a \sin az \cos 2at H}{\cos^2 az h}-\frac{h_y}{\cos az h^2}.
	\end{align*}
	Similarly
	\[\te_2 = f dy = gh dy\]
	and
	\[d \te_2 = gh_x dx \w dy + g_z h dz \w dy.\]
	On the other hand, from (\ref{dteta}) we have
	\begin{align*}
		d \te_2 &= -\G^1_{22} \te_1 \w \te_2 + \frac{1}{2} \al \te_2 \w \te_4\\
		&= -\G^1_{22} g^2 h^2 dx \w dy  + \frac{1}{2} \al (f dy) \w (dz - fk dx - fm dy)\\
		&= (-\G^1_{22} g^2 h^2 + \frac{1}{2} \al g h \sin 2at H) dx \w dy  + \frac{1}{2} \al gh dy \w dz
	\end{align*}
	and we get
	\begin{align*}
	\G^1_{22} &= \frac{1}{2} \al \frac{\sin 2at H}{gh} - \frac{h_x}{gh^2} = \frac{1}{2} \al \frac{\sin 2at H}{\cos az h} - \frac{h_x}{\cos az h^2}\\
	&= \frac{a \sin az \sin 2at H}{\cos^2 az h} - \frac{h_x}{\cos az h^2}.
	\end{align*}
	\ref{it:3} From (\ref{Ricci}) we obtain
	\begin{align*}
		Ric(E_1, E_1) &= Ric(E_2, E_2) = \frac{-4a^2 h^2 + 2a^2 H^2 -\Delta \ln h}{h^2 \cos^2 az} +6a^2,\\
		Ric(E_3, E_3) &= Ric(E_4, E_4) = 6a^2
	\end{align*}
	and $Ric(E_i,E_j)=0$ for the other cases. Hence in the frame $\{E_1,E_2,E_3,E_4\}$ we have
	\[
	Ric=\begin{bmatrix}
		\frac{-4a^2 h^2 + 2a^2 H^2 -\Delta \ln h}{h^2 \cos^2 az} +6a^2 & 0 & 0 & 0 \\
		0 & \frac{-4a^2 h^2 + 2a^2 H^2 -\Delta \ln h}{h^2 \cos^2 az} +6a^2 & 0 & 0\\
		0 & 0 & 6a^2 & 0\\
		0 & 0 & 0 & 6a^2
	\end{bmatrix}.\]\\
	\ref{it:4} The scalar curvature of $(M,g)$ is
		\[\tau = \mathrm{tr}_g Ric = \frac{-8a^2 h^2 + 4a^2 H^2 -2\Delta \ln h}{h^2 \cos^2 az} +24a^2.\]
	\ref{it:5} Since by (\ref{teta2})
	\[\vert \te \vert^2 +\delta \te = 2 E_4 \al - \al^2 = 4a^2,\]
	it follows that the conformal scalar curvature of $(M,g)$ is
		\[\kappa = \tau -6(\vert \te \vert^2 +\delta \te)= \frac{-8a^2 h^2 + 4a^2 H^2 -2\Delta \ln h}{h^2 \cos^2 az}.\]
	\ref{it:6} Since the Weyl tensor
	\[W^- =
	\begin{bmatrix}
		\frac{\kappa}{6} & 0 & 0\\
		0 & -\frac{\kappa}{12} & 0\\ 
		0 & 0 & -\frac{\kappa}{12}\\ 
	\end{bmatrix},\] we have
	\[W^- =
	\begin{bmatrix}
		\frac{-4a^2 h^2 + 2a^2 H^2 -\Delta \ln h}{3 h^2 \cos^2 az} & 0 & 0 \\
		0 & \frac{4a^2 h^2 -2a^2 H^2 +\Delta \ln h}{6 h^2 \cos^2 az} & 0\\
		0 & 0 & \frac{4a^2 h^2 -2a^2 H^2 +\Delta \ln h}{6 h^2 \cos^2 az}
	\end{bmatrix}\] in the frame $\{\Om, \phi, \psi \}$, where $\{\phi, \psi \}$ is any orthonormal basis of $LM$.\\
	\ref{it:7} From (\ref{R12}) we have
	\begin{align*}
		(R(E_1, E_2) E_1, E_2) &= -E_1 \G^1_{22} - E_2 \G^2_{11} + (\G^2_{11})^2 + (\G^1_{22})^2 + \al^2\\
		&= \frac{-2a^2 H^2 + \Delta \ln h}{g^2 h^2} + \al^2
	\end{align*}
	and from (\ref{R34}) we obtain
	\[(R(E_3, E_4) E_3, E_4) = E_3 \om^4_3 (E_4) - E_4 \om^4_3 (E_3) - \om^4_3 ((E_4 \ln \al - \al) E_3) = -4a^2.\]
	This implies that the sectional curvature of $\DE$ is
	\begin{align*}
		K(E_1,E_2)&=R(E_1, E_2, E_2, E_1) = -(R(E_1, E_2) E_1, E_2) \\
		&= \frac{2a^2 H^2 - \Delta \ln h}{g^2 h^2} - \al^2
		= \frac{2a^2 H^2 - \Delta \ln h}{\cos^2 az h^2} - 4a^2 \tan^2 az
	\end{align*}
	and of $\De$ is
	\[K(E_3,E_4)=R(E_3, E_4, E_4, E_3) = -(R(E_3, E_4) E_3, E_4) = 4a^2.\]
	\ref{it:8} From \ref{it:7} it follows that
	\[K(E_1, E_2) = K(E_3, E_4) = 4a^2\]
	if and only if
	\begin{equation}\label{deltatan}
		\Delta \ln h = -4a^2 h^2 + 2 a^2 H^2.
	\end{equation}
	It follows from \ref{it:1} that in this case we have
	\[\begin{cases}
		\Delta \ln h &= -4a^2 h^2 + 2 a^2 H^2,\\
		\Delta \ln H &= 2a^2 h^2-4a^2 H^2.
	\end{cases}\]
	If (\ref{deltatan}) holds, then $\delta=0$, $\kappa=0$ and $\tau = 24a^2$. Then from (\ref{riemann}) we obtain that the Riemann curvature tensor $R=4a^2 \Pi$ and hence the holomorphic sectional curvature is $H=4a^2$, since for $X$ such that $\vert \vert X \vert \vert = 1$ we have
	\[H(X) = K(X,\bJ X) = R(X,\bJ X,\bJ X,X) = 4a^2 \Pi (X,\bJ X,\bJ X,X) = 4a^2.\]
	It follows that (\ref{deltatan}) holds if and only if the holomorphic sectional curvature is constant and is equal to $4a^2$.
	Then we conclude that $(M,g)$ is Einstein and is a space form if and only if $\Delta \ln h = -4a^2 h^2 + 2 a^2 H^2$, because it is generalized Calabi type K\"ahler surface which is not of Calabi type and it has constant holomorphic sectional curvature. It is locally isometric to $\mathbb{CP}^2$ with its standard Fubini-Study metric of constant holomorphic sectional curvature $4a^2$ (see \cite{J-M}).
\end{proof}

Now we have the next case.
\begin{theorem}
	If metric $g$ is as in Theorem \ref{tw:coth} ($\al=-2a\coth az$, $a \in \mathbb{R}$, $a \neq 0$), then
	\begin{enumerate}[label=\upshape(\roman*), leftmargin=*, widest=iiii]
		\item $\Delta \ln H=2a^2 h^2+4a^2 H^2$ on $U$, \label{it:1}
		\item $\G^2_{11} = -\frac{a \cosh az \sin 2at H}{\sinh^2 az h}-\frac{h_y}{\sinh az h^2}$ and $\G^1_{22} = -\frac{a \cosh az \cos 2at H}{\sinh^2 az h} - \frac{h_x}{\sinh az h^2}$, \label{it:21}
		\item the Ricci tensor of $(M,g)$ represented in the frame $\{E_1,E_2,E_3,E_4\}$ is 
		\[
		Ric=\begin{bmatrix}
			-\frac{4a^2 h^2 + 2a^2 H^2 +\Delta \ln h}{h^2 \sinh^2 az} -6a^2 & 0 & 0 & 0 \\
			0 & -\frac{4a^2 h^2 + 2a^2 H^2 +\Delta \ln h}{h^2 \sinh^2 az} -6a^2 & 0 & 0\\
			0 & 0 & -6a^2 & 0\\
			0 & 0 & 0 & -6a^2
		\end{bmatrix},\]\label{it:3}
		\item the scalar curvature of $(M,g)$ is $\tau = -\frac{8a^2 h^2 + 4a^2 H^2 +2\Delta \ln h}{h^2 \sinh^2 az} -24a^2$,\label{it:4}
		\item the conformal scalar curvature of $(M,g)$ is $\kappa = -\frac{8a^2 h^2 + 4a^2 H^2 +2\Delta \ln h}{h^2 \sinh^2 az}$,\label{it:5}
		\item the Weyl tensor
		\[W^- =
		\begin{bmatrix}
			-\frac{4a^2 h^2 + 2a^2 H^2 +\Delta \ln h}{3 h^2 \sinh^2 az} & 0 & 0 \\
			0 & \frac{4a^2 h^2 +2a^2 H^2 +\Delta \ln h}{6 h^2 \sinh^2 az} & 0\\
			0 & 0 & \frac{4a^2 h^2 +2a^2 H^2 +\Delta \ln h}{6 h^2 \sinh^2 az}
		\end{bmatrix}\] in the frame $\{\Om, \phi, \psi \}$, where $\{\phi, \psi \}$ is any orthonormal basis of $LM$,\label{it:6}
		\item the sectional curvature of $(M,g)$ of $\DE$ is
		$K(E_1,E_2)= -\frac{2a^2 H^2 + \Delta \ln h}{\sinh^2 az h^2} - 4a^2 \coth^2 az$
		and of $\De$ is
		$K(E_3,E_4)= -4a^2$,\label{it:7}
		\item $(M,g)$ is a space form (it is locally isometric to $B^2$ with its standard metric of constant holomorphic sectional curvature $-4a^2$) if and only if $\Delta \ln h = -4a^2 h^2 - 2 a^2 H^2$.\label{it:8}
	\end{enumerate}
\end{theorem}
\begin{proof}
	\ref{it:1} It follows from Theorem \ref{tw:coth} that $\Delta \ln H=2a^2 h^2+4a^2 H^2$ on $U$.\\
	\ref{it:2} Take a coordinate system such that $E_1=\frac1f\p_x+k\p_z+l\p_t,E_2=\frac1f\p_y+m\p_z+n\p_t,E_3=\frac1\beta\p_t,E_4=\p_z$. Then  $\te_1=fdx, \te_2=fdy, \te_4=dz-(fk)dx-(fm)dy,\te_3=\beta dt-(\beta lf)dx-(\beta nf)dy$.  Let  $\al=-2a\coth az,\beta=\sinh 2az$. Let $g=\sinh az$. Then \[f=gh, m=\frac{\sin 2at H}{gh}, k=\frac{\cos 2at H}{gh},\]
	\[n=\frac{\coth 2az \cos 2at H -\frac{1}{2a} (\ln H)_x}{gh},
	l=\frac{-\coth 2az \sin 2at H +\frac{1}{2a} (\ln H)_y}{gh}.\]
	Since
	\[\te_1 = f dx = gh dx,\]
	we have
	\[d \te_1 = gh_y dy \w dx + g_z h dz \w dx.\]
	On the other hand, from (\ref{dteta}) we have
	\begin{align*}
		d \te_1 &= \G^2_{11} \te_1 \w \te_2 + \frac{1}{2} \al \te_1 \w \te_4\\
		&= \G^2_{11} g^2 h^2 dx \w dy  + \frac{1}{2} \al (f dx) \w (dz - fk dx - fm dy)\\
		&= (\G^2_{11} g^2 h^2 - \frac{1}{2} \al g h \sin 2at H) dx \w dy  + \frac{1}{2} \al gh dx \w dz.
	\end{align*}
	This yields
	\begin{align*}
		\G^2_{11} &= \frac{1}{2} \al \frac{\sin 2at H}{gh} - \frac{h_y}{gh^2} = \frac{1}{2} \al \frac{\sin 2at H}{\sinh az h}-\frac{h_y}{\sinh az h^2}\\
		&= -\frac{a \cosh az \sin 2at H}{\sinh^2 az h}-\frac{h_y}{\sinh az h^2}.
	\end{align*}
	Similarly
	\[\te_2 = f dy = gh dy\]
	and
	\[d \te_2 = gh_x dx \w dy + g_z h dz \w dy.\]
	On the other hand, from (\ref{dteta}) we have
	\begin{align*}
		d \te_2 &= -\G^1_{22} \te_1 \w \te_2 + \frac{1}{2} \al \te_2 \w \te_4\\
		&= -\G^1_{22} g^2 h^2 dx \w dy  + \frac{1}{2} \al (f dy) \w (dz - fk dx - fm dy)\\
		&= (-\G^1_{22} g^2 h^2 + \frac{1}{2} \al g h \cos 2at H) dx \w dy  + \frac{1}{2} \al gh dy \w dz
	\end{align*}
	and we get
	\begin{align*}
		\G^1_{22} &= \frac{1}{2} \al \frac{\cos 2at H}{gh} - \frac{h_x}{gh^2} = \frac{1}{2} \al \frac{\cos 2at H}{\sinh az h} - \frac{h_x}{\sinh az h^2}\\
		&= -\frac{a \cosh az \cos 2at H}{\sinh^2 az h} - \frac{h_x}{\sinh az h^2}.
	\end{align*}
	\ref{it:3} From (\ref{Ricci}) we obtain
	\begin{align*}
		Ric(E_1, E_1) &= Ric(E_2, E_2) = -\frac{4a^2 h^2 + 2a^2 H^2 +\Delta \ln h}{h^2 \sinh^2 az} -6a^2,\\
		Ric(E_3, E_3) &= Ric(E_4, E_4) = -6a^2
	\end{align*}
	and $Ric(E_i,E_j)=0$ for the other cases. Hence in the frame $\{E_1,E_2,E_3,E_4\}$ we have
	\[
	Ric=\begin{bmatrix}
		-\frac{4a^2 h^2 + 2a^2 H^2 +\Delta \ln h}{h^2 \sinh^2 az} -6a^2 & 0 & 0 & 0 \\
		0 & -\frac{4a^2 h^2 + 2a^2 H^2 +\Delta \ln h}{h^2 \sinh^2 az} -6a^2 & 0 & 0\\
		0 & 0 & -6a^2 & 0\\
		0 & 0 & 0 & -6a^2
	\end{bmatrix}.\]\\
	\ref{it:4} The scalar curvature of $(M,g)$ is
	\[\tau = \mathrm{tr}_g Ric = -\frac{8a^2 h^2 + 4a^2 H^2 +2\Delta \ln h}{h^2 \sinh^2 az} -24a^2.\]
	\ref{it:5} Since by (\ref{teta2})
	\[\vert \te \vert^2 +\delta \te = 2 E_4 \al - \al^2 = -4a^2,\]
	it follows that the conformal scalar curvature of $(M,g)$ is
	\[\kappa = \tau -6(\vert \te \vert^2 +\delta \te)= -\frac{8a^2 h^2 + 4a^2 H^2 +2\Delta \ln h}{h^2 \sinh^2 az}.\]
	\ref{it:6} Since the Weyl tensor
	\[W^- =
	\begin{bmatrix}
		\frac{\kappa}{6} & 0 & 0\\
		0 & -\frac{\kappa}{12} & 0\\ 
		0 & 0 & -\frac{\kappa}{12} 
	\end{bmatrix},\] we have
	\[W^- =
	\begin{bmatrix}
		-\frac{4a^2 h^2 + 2a^2 H^2 +\Delta \ln h}{3 h^2 \sinh^2 az} & 0 & 0 \\
		0 & \frac{4a^2 h^2 +2a^2 H^2 +\Delta \ln h}{6 h^2 \sinh^2 az} & 0\\
		0 & 0 & \frac{4a^2 h^2 +2a^2 H^2 +\Delta \ln h}{6 h^2 \sinh^2 az}
	\end{bmatrix}\] in the frame $\{\Om, \phi, \psi \}$, where $\{\phi, \psi \}$ is any orthonormal basis of $LM$.\\
	\ref{it:7} From (\ref{R12}) we have
	\begin{align*}
		(R(E_1, E_2) E_1, E_2) &= -E_1 \G^1_{22} - E_2 \G^2_{11} + (\G^2_{11})^2 + (\G^1_{22})^2 + \al^2\\
		&= \frac{2a^2 H^2 + \Delta \ln h}{g^2 h^2} + \al^2
	\end{align*}
	and from (\ref{R34}) we obtain
	\[(R(E_3, E_4) E_3, E_4) = E_3 \om^4_3 (E_4) - E_4 \om^4_3 (E_3) - \om^4_3 ((E_4 \ln \al - \al) E_3) = 4a^2.\]
	This implies that the sectional curvature of $\DE$ is
	\begin{align*}
		K(E_1,E_2)&=R(E_1, E_2, E_2, E_1) = -(R(E_1, E_2) E_1, E_2)\\
		&= -\frac{2a^2 H^2 + \Delta \ln h}{g^2 h^2} - \al^2
		=-\frac{2a^2 H^2 + \Delta \ln h}{\sinh^2 az h^2} - 4a^2 \coth^2 az
	\end{align*}
	and of $\De$ is
	\[K(E_3,E_4)=R(E_3, E_4, E_4, E_3) = -(R(E_3, E_4) E_3, E_4) = -4a^2.\]
	\ref{it:8} From \ref{it:7} it follows that
	\[K(E_1, E_2) = K(E_3, E_4) = -4a^2\]
	if and only if
	\begin{equation}\label{deltacoth}
		\Delta \ln h = -4a^2 h^2 - 2 a^2 H^2.
	\end{equation}
	It follows from \ref{it:1} that in this case we have
	\[\begin{cases}
		\Delta \ln h &= -4a^2 h^2 - 2 a^2 H^2,\\
		\Delta \ln H &= 2a^2 h^2+4a^2 H^2.
	\end{cases}\]
	If (\ref{deltacoth}) holds, then $\delta=0$, $\kappa=0$ and $\tau = -24a^2$. Then from (\ref{riemann}) we obtain that the Riemann curvature tensor $R=-4a^2 \Pi$ and hence the holomorphic sectional curvature is $H=-4a^2$, since for $X$ such that $\vert \vert X \vert \vert = 1$ we have
	\[H(X) = K(X,\bJ X) = R(X,\bJ X,\bJ X,X) = -4a^2 \Pi (X,\bJ X,\bJ X,X) = -4a^2.\]
	It follows that (\ref{deltacoth}) holds if and only if the holomorphic sectional curvature is constant and is equal to $-4a^2$.
	Then we conclude that $(M,g)$ is Einstein and is a space form if and only if $\Delta \ln h = -4a^2 h^2 - 2 a^2 H^2$, because it is generalized Calabi type K\"ahler surface which is not of Calabi type and it has constant holomorphic sectional curvature. It is locally isometric to $B^2$ with its standard metric of constant holomorphic sectional curvature $-4a^2$ (see \cite{J-M}).
\end{proof}

Finally, we have the last case.
\begin{theorem}
	If metric $g$ is as in Theorem \ref{tw:tanh} ($\al=-2a\tanh az$, $a \in \mathbb{R}$, $a \neq 0$), then
	\begin{enumerate}[label=\upshape(\roman*), leftmargin=*, widest=iiii]
		\item $\Delta \ln H=-2a^2 h^2+4a^2 H^2$ on $U$, \label{it:1}
		\item $\G^2_{11} = -\frac{a \sinh az \cos 2at H}{\cosh^2 az h}-\frac{h_y}{\cosh az h^2}$ and $\G^1_{22} = -\frac{a \sinh az \sin 2at H}{\cosh^2 az h} - \frac{h_x}{\cosh az h^2}$, \label{it:2}
		\item the Ricci tensor of $(M,g)$ represented in the frame $\{E_1,E_2,E_3,E_4\}$ is 
		\[
		Ric=\begin{bmatrix}
			\frac{4a^2 h^2 - 2a^2 H^2 -\Delta \ln h}{h^2 \cosh^2 az} -6a^2 & 0 & 0 & 0 \\
			0 & \frac{4a^2 h^2 - 2a^2 H^2 -\Delta \ln h}{h^2 \cosh^2 az} -6a^2 & 0 & 0\\
			0 & 0 & -6a^2 & 0\\
			0 & 0 & 0 & -6a^2
		\end{bmatrix},\]\label{it:3}
		\item the scalar curvature of $(M,g)$ is $\tau = \frac{8a^2 h^2 - 4a^2 H^2 -2\Delta \ln h}{h^2 \cosh^2 az} -24a^2$,\label{it:4}
		\item the conformal scalar curvature of $(M,g)$ is $\kappa = \frac{8a^2 h^2 - 4a^2 H^2 -2\Delta \ln h}{h^2 \cosh^2 az}$,\label{it:5}
		\item the Weyl tensor
		\[W^- =
		\begin{bmatrix}
			\frac{4a^2 h^2 - 2a^2 H^2 -\Delta \ln h}{3 h^2 \cosh^2 az} & 0 & 0 \\
			0 & \frac{-4a^2 h^2 +2a^2 H^2 +\Delta \ln h}{6 h^2 \cosh^2 az} & 0\\
			0 & 0 & \frac{-4a^2 h^2 +2a^2 H^2 +\Delta \ln h}{6 h^2 \cosh^2 az}
		\end{bmatrix}\] in the frame $\{\Om, \phi, \psi \}$, where $\{\phi, \psi \}$ is any orthonormal basis of $LM$, \label{it:6}
		\item the sectional curvature of $(M,g)$ of $\DE$ is
		$K(E_1,E_2)= -\frac{2a^2 H^2 + \Delta \ln h}{\cosh^2 az h^2} - 4a^2 \tanh^2 az$
		and of $\De$ is
		$K(E_3,E_4)= -4a^2$,\label{it:7}
		\item $(M,g)$ is a space form (it is locally isometric to $B^2$ with its standard metric of constant holomorphic sectional curvature $-4a^2$) if and only if $\Delta \ln h = 4a^2 h^2 - 2 a^2 H^2$.\label{it:8}
	\end{enumerate}
\end{theorem}
\begin{proof}
	\ref{it:1} It follows from Theorem \ref{tw:tanh} that $\Delta \ln H=-2a^2 h^2+4a^2 H^2$ on $U$.\\
	\ref{it:2} Take a coordinate system such that $E_1=\frac1f\p_x+k\p_z+l\p_t,E_2=\frac1f\p_y+m\p_z+n\p_t,E_3=\frac1\beta\p_t,E_4=\p_z$. Then  $\te_1=fdx, \te_2=fdy, \te_4=dz-(fk)dx-(fm)dy,\te_3=\beta dt-(\beta lf)dx-(\beta nf)dy$.  Let  $\al=-2a\tanh az,\beta=\sinh 2az$. Let $g=\cosh az$. Then \[f=gh, m=\frac{\cos 2at H}{gh}, k=\frac{\sin 2at H}{gh},\]
	\[n=\frac{-\coth 2az \sin 2at H +\frac{1}{2a} (\ln H)_x}{gh},
	l=\frac{\coth 2az \cos 2at H -\frac{1}{2a} (\ln H)_y}{gh}.\]
	Since
	\[\te_1 = f dx = gh dx,\]
	we have
	\[d \te_1 = gh_y dy \w dx + g_z h dz \w dx.\]
	On the other hand, from (\ref{dteta}) we have
	\begin{align*}
		d \te_1 &= \G^2_{11} \te_1 \w \te_2 + \frac{1}{2} \al \te_1 \w \te_4\\
		&= \G^2_{11} g^2 h^2 dx \w dy  + \frac{1}{2} \al (f dx) \w (dz - fk dx - fm dy)\\
		&= (\G^2_{11} g^2 h^2 - \frac{1}{2} \al g h \cos 2at H) dx \w dy  + \frac{1}{2} \al gh dx \w dz.
	\end{align*}
	This yields
	\begin{align*}
		\G^2_{11} &= \frac{1}{2} \al \frac{\cos 2at H}{gh} - \frac{h_y}{gh^2} = \frac{1}{2} \al \frac{\cos 2at H}{\cosh az h}-\frac{h_y}{\cosh az h^2}\\
		&= -\frac{a \sinh az \cos 2at H}{\cosh^2 az h}-\frac{h_y}{\cosh az h^2}.
	\end{align*}
	Similarly
	\[\te_2 = f dy = gh dy\]
	and
	\[d \te_2 = gh_x dx \w dy + g_z h dz \w dy.\]
	On the other hand, from (\ref{dteta}) we have
	\begin{align*}
		d \te_2 &= -\G^1_{22} \te_1 \w \te_2 + \frac{1}{2} \al \te_2 \w \te_4\\
		&= -\G^1_{22} g^2 h^2 dx \w dy  + \frac{1}{2} \al (f dy) \w (dz - fk dx - fm dy)\\
		&= (-\G^1_{22} g^2 h^2 + \frac{1}{2} \al g h \sin 2at H) dx \w dy  + \frac{1}{2} \al gh dy \w dz
	\end{align*}
	and we get
	\begin{align*}
		\G^1_{22} &= \frac{1}{2} \al \frac{\sin 2at H}{gh} - \frac{h_x}{gh^2} = \frac{1}{2} \al \frac{\sin 2at H}{\cosh az h} - \frac{h_x}{\cosh az h^2}\\
		&= -\frac{a \sinh az \sin 2at H}{\cosh^2 az h} - \frac{h_x}{\cosh az h^2}.
	\end{align*}
	\ref{it:3} From (\ref{Ricci}) we obtain
	\begin{align*}
		Ric(E_1, E_1) &= Ric(E_2, E_2) = \frac{4a^2 h^2 - 2a^2 H^2 -\Delta \ln h}{h^2 \cosh^2 az} -6a^2,\\
		Ric(E_3, E_3) &= Ric(E_4, E_4) = -6a^2
	\end{align*}
	and $Ric(E_i,E_j)=0$ for the other cases. Hence in the frame $\{E_1,E_2,E_3,E_4\}$ we have
	\[
	Ric=\begin{bmatrix}
		\frac{4a^2 h^2 - 2a^2 H^2 -\Delta \ln h}{h^2 \cosh^2 az} -6a^2 & 0 & 0 & 0 \\
		0 & \frac{4a^2 h^2 - 2a^2 H^2 -\Delta \ln h}{h^2 \cosh^2 az} -6a^2 & 0 & 0\\
		0 & 0 & -6a^2 & 0\\
		0 & 0 & 0 & -6a^2
	\end{bmatrix}.\]\\
	\ref{it:4} The scalar curvature of $(M,g)$ is
	\[\tau = \mathrm{tr}_g Ric = \frac{8a^2 h^2 - 4a^2 H^2 -2\Delta \ln h}{h^2 \cosh^2 az} -24a^2.\]
	\ref{it:5} Since by (\ref{teta2})
	\[\vert \te \vert^2 +\delta \te = 2 E_4 \al - \al^2 = -4a^2,\]
	it follows that the conformal scalar curvature of $(M,g)$ is
	\[\kappa = \tau -6(\vert \te \vert^2 +\delta \te)= \frac{8a^2 h^2 - 4a^2 H^2 -2\Delta \ln h}{h^2 \cosh^2 az}.\]
	\ref{it:6} Since the Weyl tensor
	\[W^- =
	\begin{bmatrix}
		\frac{\kappa}{6} & 0 & 0\\
		0 & -\frac{\kappa}{12} & 0\\ 
		0 & 0 & -\frac{\kappa}{12}\\ 
	\end{bmatrix},\] we have
	\[W^- =
	\begin{bmatrix}
		\frac{4a^2 h^2 - 2a^2 H^2 -\Delta \ln h}{3 h^2 \cosh^2 az} & 0 & 0 \\
		0 & \frac{-4a^2 h^2 +2a^2 H^2 +\Delta \ln h}{6 h^2 \cosh^2 az} & 0\\
		0 & 0 & \frac{-4a^2 h^2 +2a^2 H^2 +\Delta \ln h}{6 h^2 \cosh^2 az}
	\end{bmatrix}\] in the frame $\{\Om, \phi, \psi \}$, where $\{\phi, \psi \}$ is any orthonormal basis of $LM$.\\
	\ref{it:7} From (\ref{R12}) we have
	\begin{align*}
		(R(E_1, E_2) E_1, E_2) &= -E_1 \G^1_{22} - E_2 \G^2_{11} + (\G^2_{11})^2 + (\G^1_{22})^2 + \al^2\\
		&= \frac{2a^2 H^2 + \Delta \ln h}{g^2 h^2} + \al^2
	\end{align*}
	and from (\ref{R34}) we obtain
	\[(R(E_3, E_4) E_3, E_4) = E_3 \om^4_3 (E_4) - E_4 \om^4_3 (E_3) - \om^4_3 ((E_4 \ln \al - \al) E_3) = 4a^2.\]
	This implies that the sectional curvature of $\DE$ is
	\begin{align*}
		K(E_1,E_2)&=R(E_1, E_2, E_2, E_1) = -(R(E_1, E_2) E_1, E_2)\\
		&= -\frac{2a^2 H^2 + \Delta \ln h}{g^2 h^2} - \al^2
		=-\frac{2a^2 H^2 + \Delta \ln h}{\cosh^2 az h^2} - 4a^2 \tanh^2 az
	\end{align*}
	and of $\De$ is
	\[K(E_3,E_4)=R(E_3, E_4, E_4, E_3) = -(R(E_3, E_4) E_3, E_4) = -4a^2.\]
	\ref{it:8} From \ref{it:7} it follows that
	\[K(E_1, E_2) = K(E_3, E_4) = -4a^2\]
	if and only if
	\begin{equation}\label{deltatanh}
		\Delta \ln h = 4a^2 h^2 - 2 a^2 H^2.
	\end{equation}
	It follows from \ref{it:1} that in this case we have
	\[\begin{cases}
		\Delta \ln h &= 4a^2 h^2 - 2 a^2 H^2,\\
		\Delta \ln H &= -2a^2 h^2+4a^2 H^2.
	\end{cases}\]
	If (\ref{deltatanh}) holds, then $\delta=0$, $\kappa=0$ and $\tau = -24a^2$. Then from (\ref{riemann}) we obtain that the Riemann curvature tensor $R=-4a^2 \Pi$ and hence the holomorphic sectional curvature is $H=-4a^2$, since for $X$ such that $\vert \vert X \vert \vert = 1$ we have
	\[H(X) = K(X,\bJ X) = R(X,\bJ X,\bJ X,X) = -4a^2 \Pi (X,\bJ X,\bJ X,X) = -4a^2.\]
	It follows that (\ref{deltatanh}) holds if and only if the holomorphic sectional curvature is constant and is equal to $-4a^2$.
	Then we conclude that $(M,g)$ is Einstein and is a space form if and only if $\Delta \ln h = 4a^2 h^2 - 2 a^2 H^2$, because it is generalized Calabi type K\"ahler surface which is not of Calabi type and it has constant holomorphic sectional curvature. It is locally isometric to $B^2$ with its standard metric of constant holomorphic sectional curvature $-4a^2$ (see \cite{J-M}). 
\end{proof}

\begin{remark}
	Note that in every case of generalized Calabi type K\"ahler surface which is not of Calabi type we have $W^-=0$ if and only if $(M,g)$ is a space form. If $(M,g)$ is not a space form, then $W^- \neq 0$.
\end{remark} 

\begin{remark}
	Note that we get examples of local hermitian structures, which give opposite orientation, which are not locally conformally K\"ahler on space forms $\mathbb{CP}^2$, $\mathbb{C}^2$ and $B^2$.\\
\end{remark}

\textbf{\large{Acknowledgements}}

The author is grateful to Professor Włodzimierz Jelonek for his advise and help during preparations of this paper.

\bigskip
\textbf{Author's adress:}\\
Department of Applied Mathematics\\
Cracow  University of Technology\\
Warszawska 24\\
31-155 Krak\'ow, POLAND\\
\emph{Email address: ewelina.mulawa@pk.edu.pl}

\end{document}